\documentclass[Unicode]{amsart}
% ...
\usepackage{destemacros}
\usepackage{enumerate}

\title
{A theorem on support $\tau$-tilting pairs}

\author[{G.}  D'Este]{{Gabriella} {D'Este}}

\address{University of Milano\\
Department of Mathematics ``Federigo Enriques''\\
Via Cesare Saldini, 50 \\
20133 Milano MI\\
(Italy)}

\email{gabriella.deste@unimi.it}

\thanks{Dedicated to Professor Claus Michael Ringel on his 80th birthday.}



\begin{document}
\begin{abstract}
  We show that there is a special bijection between the indecomposable summands of the two modules which form a basic support $\tau$-tilting pair and the indecomposable summands of the two modules which form another basic support $\tau$-tilting pair. 
\end{abstract}

\maketitle

\section{Introduction and definitions}

Given two ``basic'' pairs of modules $\tonde{T, P}$ and $\tonde{T^\prime, P^{\prime}}$, where $T$ and $T^{\prime}$ do not admit non zero morphisms to $\tau T$ and $\tau T^{\prime}$ respectively while $P$ and $P^{\prime}$ are projective, we construct a bijection between the indecomposable summands of $T\oplus P$ and those of $T^{\prime} \oplus P^{\prime}$. 
If this bijection sends $X_{i}$ to $Y_{j}$, then either $X_{i}$ is isomorphic to $Y_{j}$, or $X_{i} \oplus Y_{j}$ is 
not ``$\tau$-rigid'' or one of the pairs $\left(X_{i}, Y_{j}\right)$ and $\left(Y_{j}, X_{i}\right)$ is not ``$\tau$-rigid''. 
In this way we extend a bijection between the indecomposable summands of two basic tilting modules \cite[Theorem 2.5]{4}, or between the indecomposable summands of two basic $\tau$-tilting modules \cite[Acknowledgment]{5}.

In the following we recall some definitions and we fix some conventions. First of all, $A$ will denote a basic \cite[page 35]{3} finite dimensional algebra over the field $K$, and $n$ will be the number of isomorphism classes of simple $A$-modules. Moreover,
modules will be left A-modules, and $\tau$ will denote the Auslander--Reiten translation \cite[page 225]{3}. 

Finally, if $M$ is a finitely generated module, then $\add M$ will demote the class of all finite direct sums of direct summands of $M$.
Let $T$ be a finitely generated $A$-module. Then we say that $T$ is \emph{$\tau$-rigid} if $\Hom_{A}(T, \tau T)=0$. 
Next, we say that $T$ is \emph{$\tau$-tilting} if $T$ is $\tau$-rigid and
$n$ is the number of isomorphism classes of the indecomposable direct summands of $M$. 
%
%
%
We know from \cite[Proposition 1.3]{1} or (\cite[Proposition 3.6]{8} and \cite{2}) that, if $M$ is a basic and $\tau$-rigid module, then $M$ is the direct sum of at most $n$ indecomposable summands. Now, let $P$ be a finitely generated projective module. Then we say that the pair $(T, P)$ is a \emph{$\tau$-rigid pair} if $T$ is a $\tau$-rigid module and $\Hom_{A}\left(P, T\right)=0$. Moreover, we say that $(T, P)$ is a \emph{support $\tau$-tilting pair} if $\left(T, P\right)$ is a $\tau$-rigid pair such that $n$ is the number of is isomorphism classes of the indecomposable direct summands of $T \oplus P$. We say that a support $\tau$-tilting pair $(T, P)$ is \emph{basic} \cite[Definition 4.7]{6} if both the modules $T$ and $P$, and so also $T \oplus P$, are basic modules. 
We also recall that a module $T$ is a \emph{partial tilting module} if the projective dimension of $T$ is at most one and $\Ext_{A}^{1}\tonde{T, T}=0$. Moreover, we say that a partial tilting module $T$ is a \emph{tilting module} if $n$ is the number of the isomorphism classes of the indecomposable summands of $T$. Using this terminology,
given two basic  tilting (resp. $\tau$-tilting) modules $T=\oplus_{i=1}^{n} X_{i}$ and $T^{\prime}=\oplus_{i=1}^{n} Y_{i}$, we know from 
\cite[Theorem 2.5]{4} (resp. \cite[Acknowledgment]{5}) that there is a permutation $s \in S_{n}$ such that the bijection $X_{i} \mapsto Y_{s(i)}$ has the following properties:
\begin{itemize}
  \item If $X_{i} \simeq Y_{j}$ for some $j$, then we have $s(i)=j$.
  \item If $X_{i} \neq Y_{j}$ for any $j$, then the module $X_{i} \oplus Y_{s(i)}$ is not partial tilting (resp. is not $\tau$-rigid).\\
\end{itemize}

Finally, we recall the relationship between some of the previous modules or pairs.

\begin{lem}\cite[page 167]{7}
Let $T$ be a tilting module, and let $X$ be a module such that $T \oplus X$ is a partial tilting module. Then $X \in \add T$.
\end{lem}

\begin{lem}\cite[Theorem 2.12]{1}
Let $T$ be a $\tau$-tilting module, and let $X$ be a module such that $T \oplus X$ is a $\tau$-rigid module. Then $X \in \add T$.
\end{lem}

\begin{lem}\cite[Corollary 2.13]{1}
 Let $(T, P)$ be a support $\tau$-tilting pair, and let $X$ be a module such that $(T \oplus X, P)$ is a $\tau$-rigid pair. Then $X \in$ add $T$.
\end{lem}

\section{A theorem and some examples}

Keeping the definitions and the notation of the introduction, we prove the following theorem.
\begin{thm} 
Let $(T, P)=\tonde{X_{1} \oplus \dots \oplus X_{r}, X_{r+1} 
\oplus \dots \oplus X_{n}}$ and 
$\tonde{T^{\prime}, P^{\prime}}=$\\ $\tonde{Y_{1} \oplus \dots \oplus Y_{r'}, Y_{r'+1} \oplus \dots \oplus Y_{n}}$ 
be two basic support $\tau$-tilting pairs such that the modules $X_{i}$ and $Y_{i}$ are indecomposable for any $i$. Then there is a permutation $s\in S_{n}$ such that the bijection $\set{X_{1},\dots,X_{n}}\to\set{Y_{1},\dots,Y_{n}}$, sending $X_{i}$ to $Y_{s(i)}$ for any $i$, has the property that for any $i$ the two modules $X_{i}$ and $Y_{s(i)}$ satisfy one of the following conditions: 
\begin{enumerate}[(a)]
	\item $X_{i}$ is isomorphic to $Y_{s(i)}$. 
	\item $X_{i}\oplus Y_{s(i)}$ is not $\tau$-rigid. 
	\item $Y_{s(i)}\in \add P'$ but the pair $\tonde{X_{i},Y_{s(i)}}$ is not $\tau$-rigid. 
	\item $X_{i}\in\add P$ but the pair $\tonde{Y_{s(i)},X_{i}}$ is not $\tau$-rigid. 
\end{enumerate}
\end{thm}
\begin{proof} 
For any $i=1, \ldots, n$, let $F(i)$ be the subset formed by all $j \in\{1, \ldots, n\}$ such that one of the following conditions holds:
\begin{enumerate}[(a$^{\star}$)]
\item $X_{i}$ is isomorphic to $Y_{j}$.
\item $X_{i} \oplus Y_{j}$ is not $\tau$-rigid. 
\item $Y_{j} \in \add P'$ but the pair $\left(X_{i}, Y_{j}\right)$ is not $\tau$-rigid.
\item $X_{i} \in\add P$ but the pair $\left(Y_{j}, X_{i}\right)$ is not $\tau$-rigid. 
\end{enumerate}
Suppose, by contradiction, that there is an $i$ such that $F(i)=\varnothing$. 
Then we deduce from $(a^{\star})$ that:
\begin{enumerate}[(1)]
\item $X_{i}, Y_{1}, \ldots, Y_{n}$ are pairwise non isomorphic.\\ On the other hand, we deduce from $(b^{\star})$  that:
\item The module $T^{\prime} \oplus X_{i}=Y_{1} \oplus \ldots \oplus Y_{r^{\prime}} \oplus X_{i}$ is $\tau$-rigid.\\
Moreover, we deduce from $(c^{\star})$ that:
\item $\operatorname{Hom}_{A}\left(P^{\prime}, X_{i}\right)=0$.\\
Finally, we deduce from $(d^{\star})$ that:
\item If $X_{i} \in \add P$, then $\Hom_{A}\left(X_{i}, T^{\prime} \oplus P^{\prime}\right)=0$.
\end{enumerate}
We first show that if $X_{i} \in\add P$, then $\Hom_{A}\left(X_{i}, T^{\prime} \oplus P^{\prime}\right) \neq 0$. Indeed, if \\ $\Hom_{A}\left(X_{i}, T^{\prime}\oplus P^{\prime}\right)=0$, 
then $\left(T^{\prime}, P^{\prime} \oplus X_{i}\right)$ is a $\tau$-rigid pair. 
Hence, we know from \cite[Remark after Definition 2.1]{8} that $T'\oplus P' \oplus
X_{i}$ is a direct sum of at most $n$ indecomposable modules, up to multiplicity. Since $(T^{\prime}, P^{\prime})$ is a basic support $\tau$-tilting pair, it follows
that $T'\oplus P'$ is the direct sum of $n$ pairwise non isomorphic indecomposable modules. This implies that $X_{i}\in\add(T'\oplus P')$. Hence we have $X_{i}\simeq Y_{j}$ for some $j$. This remark and $(a^\star)$ imply that $j\in F(i)$, and so $F(i)\neq \emptyset$. This is a contradiction. 
Therefore, we may assume that $X_{i} \in \add T$. Since $(T',P')$ is a support $\tau$-tilting pair, we have $\Hom_{A}(P',T')=0$. This remark and (3) imply that 
\begin{enumerate}[(5)]
\item $\Hom_{A}(P',T'\oplus X_{i})=0$. 
\end{enumerate}
Putting (2) and (5) together, we conclude that $(T'\oplus X_{i},P')$ is a $\tau$-rigid pair. Since $(T',P')$ is a support $\tau$-tilting pair, Lemma 3 implies that $X_{i}\in\add T'$. This is a contradiction to (1). This contradiction implies that 
\begin{enumerate}[(6)]
\item $F(i)\neq \emptyset$ for any $i$. 
\end{enumerate}
Next, we claim that if $m=2, \ldots, n$ and $1 \leq i_{1}<i_{2}<\ldots<i_{m} \leq n$, then we have $\left|F\left(i_{1}\right) \cup \ldots \cup F\left(i_{m}\right)\right| \geq  m$.


Since $|F(i)| \geq 1$ for any $i$, we may assume by induction that there is some integer $m$ such that $2\leq m\leq n$ and the following property holds:
\begin{enumerate}[(7)]
\item If $t=1,\dots,m-1$ and $1\leq i_i <\dots<i_t\leq n$, then $\left|F(i_1)\cup\dots\cup F(i_t)\right| \geq t$. 
\end{enumerate}
Assume now, by contradiction, that 
$\abs{F(j_1)\cup\dots\cup F(j_m)} < m$ for some $1\leq j_1<\dots<j_m\leq n$. Then we deduce from (7) that 
\begin{enumerate}[(8)]
\item $\abs{F(j_1)\cup\dots\cup F(j_m)} = m-1$. 
\end{enumerate}
Let $B$ and $C$ be the subsets of $\set{1,\dots,n}\setminus \tonde{F(j_1)\cup\dots\cup F(j_m) }$ such that $b\in B$ iff $Y_b\in\add T'$ and $c\in C$ iff $Y_c\in \add P'$. Next, let $D$ and $E$ be the subsets of $J=\set{j_1,\dots,j_m}$ such that $d\in D$ iff $X_d\in\add T$ and $e\in E$ iff $X_e\in\add P$. Now, let $U = \oplus_{b\in B} Y_b$, $V = \oplus_{c\in C} Y_c$, $W = \oplus_{d\in D} X_d$, and $Z = \oplus_{e\in E} X_e$. Since $(U,V)$ is a direct summand of $(T',P')$ and $(W,Z)$ is a direct summand of $(T,P)$, the following facts hold: 
\begin{enumerate}[(9)]
\item $(U,V)$ is a $\tau$--rigid pair and $U\oplus V$ is a basic direct sum of $n-(m-1)$ indecomposable modules. 
\end{enumerate}
\begin{enumerate}[(10)]
\item $(W,Z)$ is a $\tau$--rigid pair and $W\oplus Z$ is a basic direct sum of $m$ indecomposable modules. 
\end{enumerate}
Moreover, we deduce from $(a^\star)$ and $(b^\star)$ that 
\begin{enumerate}[label=(11)]
\item $Y_b\not \simeq X_d$ and $Y_b\oplus X_d$ is $\tau$--rigid for any $b\in B$ and $d\in D$. 
\end{enumerate}
Since $U$ and $W$ are $\tau$--rigid, this implies that 
\begin{enumerate}[(12)]
\item $U\oplus W$ is $\tau$--rigid. 
\end{enumerate}
On the other hand, we deduce from $(c^\star)$ [resp. $(d^\star)$] that $\Hom_A(Y_c,X_d)=0$ for any $c\in C$ and $d\in D$ [resp. $\Hom_A(X_e,Y_b)=0$ for any $e\in E$ and $b\in B$. 
Hence, we deduce from (9) and (10) that
\begin{enumerate}[(13)]
\item $V\oplus Z$ is projective and $\Hom_A(V\oplus Z, U\oplus W)=0$. 
\end{enumerate}
Putting (12) and (13) together, we conclude 
\begin{enumerate}[(14)]
\item $(U\oplus W,V\oplus Z)$ is a basic $\tau$--rigid pair. 
\end{enumerate}
To find the desired contradiction, we deduce form (9) and (10) that $U\oplus V\oplus W\oplus Z$ is the direct sum of $n-(m-1)+m$ indecomposable modules. 
However, we know from \cite[Remark after Definition 2.1]{8} that if $(T",P")$ is a basic $\tau$--rigid pair, then $T"\oplus P"$ has at most $n$ indecomposable summands. This contradiction implies that
\begin{enumerate}[(15)]
\item $\abs{F(j_1)\cup\dots\cup F(j_m)}\geq m$ if $1\leq j_1 <\dots<j_m\leq n$. 
\end{enumerate}
Consequently, we deduce from (7) and (15) that 
\begin{enumerate}[(16)]
\item $\abs{F(i_1)\cup\dots\cup F(i_t)}\geq t$ for any $t=1,\dots,n$ and $1\leq i_1 <\dots<i_t\leq n$. 
\end{enumerate}
Putting (6) and (15) together, we conclude that the subsets $F(1), \ldots, F(n)$ satisfy the hypotheses of \cite[Lemma 2.3]{4}. Hence, we deduce from \cite[Lemma 2.3]{4} that there is a permutation $s \in S_{n}$ such that $s(i) \in F(i)$ for any $i$.
\end{proof}
The next example shows that we cannot delete condition $(c)$ in the hypotheses of Theorem 2.1. 
\begin{ex} Keeping all the notation of Theorem 2.1 , let $(T, P)$ and $\left(T^{\prime}, P^{\prime}\right)$ be two basic support $\tau$-tilting pairs of the form $\left(T, P\right)=\left(X_{1} \oplus X_{2}, X_{3}\right)$ and $\left(T^{\prime}, P^{\prime}\right)=\left(Y_{1} \oplus Y_{2}, Y_{3}\right)$.

For any $i=1,2,3$, let $G(i)$ be the subset of $F(i)$ formed by all $j$ such that one of the following conditions holds:
\begin{enumerate}[(a)]
\item  $X_{i}$ is isomorphic to $Y_{j}$.
\item $X_{i} \oplus Y_{j}$ is not $\tau$-rigid.
\item[(d)] $X_{i} \in \add P$ but $\left(Y_{j}, X_{i}\right)$ is not a $\tau$-rigid pair.
\end{enumerate} 
Then we may have $G(1)=\{1,2,3\}, G(2)=\emptyset$ and $G(3)=\{1,2\}$.
\end{ex}

\begin{proof}[Construction] Let $A$ be the algebra given by the quiver $\raisebox{5mm}{\scalebox{0.6}{\xymatrix{ & \coluno{1}\ar[ld] \ar[dr] &\\ \coluno{2} & & \coluno{3}}}}$. Next, let $X_{1}=2, X_{2}= \coldue{1}{2}$, $X_{3}=3, Y_{1}=\coldue{1}{3}, Y_{2}=1$ and $Y_{3}=2$. Then 
the following facts hold:

\begin{enumerate}[(1)]
\item  $X_{2} \oplus Y_{2}=\coldue{1}{2} \oplus \coluno{1}$ is $\tau$-rigid, $X_{2} \in\add T$ and $Y_{2} \in \add T^{\prime}.$
\item $X_{2} \oplus Y_{1}=\coldue{1}{2} \oplus \coldue{1}{3}$ is $\tau$-rigid, $X_{2} \in \add T$ and $Y_{1} \in \add T^{\prime}$.
\item $X_{2} \oplus Y_{3}=\coldue{1}{2} \oplus \coluno{2}$ is $\tau$-rigid, $X_{2} \in \add T$, $Y_{3}=P^{\prime}$ and the pair $\left(X_{2}, Y_{3}\right)=\left(\coldue{1}{2}, \coluno{2}\right)$ is not $\tau$-rigid.
\end{enumerate} 
Consequently, we deduce from (1), (2) and (3) that $F(2)=\set{3}$ and $G(2)= \emptyset$. 
Moreover, the following facts hold:
\begin{enumerate}[(1)]
\setcounter{enumi}{3}
\item $X_{1} \oplus Y_{1}=2 \oplus \coldue{1}{3}$ and $X_{1} \oplus Y_{2}=2 \oplus 1$ are not 
$\tau$-rigid, and we have $X_{1}=2=Y_{3}$.
\item $X_{3}=3=P$ and $\tonde{Y_{1}, X_{3}} =\tonde{\coldue{1}{3}, 3}$ is not $\tau$-rigid.
\item $X_{3} \oplus Y_{2}=3 \oplus 1$ is not $\tau$-rigid.
\item $X_{3} \oplus Y_{3}=3 \oplus 2$ is a $\tau$-rigid module and the pairs $\left(X_{3}, Y_{3}\right)=(3,2)$ and $\left(Y_{3}, X_{3}\right)=(2,3)$ are $\tau$-rigid pairs.
\end{enumerate}
Hence, we deduce from (4) that $G(1)=\{1,2,3\}=F(1)$. On the other hand, putting (5), (6) and (7) together, we conclude that $G(3)=\{1,2\}=F(3)$.
\end{proof}

\begin{cor}
With all the notation of Example 2.2 the modules $T$ and $P^{\prime}$ have a common projective summand, namely $X_{1}=2=Y_{3}$, such that $Y_{1}$ and $Y_{2}$ are the summands of the form $Y_{s(1)}$ for some permutation $s$ satisfying the hypotheses of Theorem 2.1.
\end{cor}
\begin{proof} Since $F(2)=\{3\}$, it follows that $s(1) \neq 3$ for any permutation $s$ as in
 Theorem 2.1. Consequently, the bijection $X_{i} \longmapsto Y_{s(i)}$ does not fix $X_{1}=2=Y_{3}$,
which is the common summand of $\bigoplus_{i=1}^{3} X_{i}$ and $\bigoplus_{i=1}^{3} Y_{i}$. Since $F(1) = \set{1,2,3}$, $F(2)=\set{3}$
%
%
and $F(3)=\{1,2\}$, it follows that the permutations satisfying the hypotheses of Theorem 2.1 are $a=(123)$ and $b=(23)$. 
\end{proof}
The next example shows that we cannot delete condition $(d)$ in the hypotheses of Theorem 2.1.
\begin{ex}
Keeping all the notation of Theorem 2.1, let $(T, P)$ and $\left(T^{\prime}, P^{\prime}\right)$ be basic support $r$-tilting pairs of the from $(T, P)=\left(X_{1} \oplus X_{2}, X_{3}\right)$ and $\tonde{T', P'} =\tonde{Y_{1} \oplus Y_{2}, Y_{3}}$. For $i=1,2,3$, let $H(i)$ be the subset of $F(i)$ formed by all $j$ such that one of the following conditions holds:
\begin{enumerate}[(a)]
\item $X_{i}$ is isomorphic to $Y_{j}$.
\item $X_{i} \oplus Y_{j}$ is not $\tau$-rigid.
\item $Y_{j} \in$ add $P^{\prime}$ but the pair $\left(X_{i}, Y_{j}\right)$ is not $\tau$-rigid.
\end{enumerate}
Then we may have $H(i) \neq \emptyset$ for any $i$ but $H(1) \cup H(2) \cup H(3) \neq\{1,2,3\}$.
\end{ex}
\begin{proof}[Construction]
Let $A$ be as in Example 2.2, and let $X_{1}=\coldue{1}{3}, X_{2}=1, X_{3}=2, Y_{1}=2, Y_{2}= \coldue{1}{2}$ and $Y_{3}=3$. Then the following facts hold:
\begin{enumerate}[(1)]
\item $X_{1} \oplus Y_{2}=\coldue{1}{3} \oplus \coldue{1}{2}$ is $\tau$-rigid, $X_{1} \in \add T$ and $Y_{2} \in\add T^{\prime}$.
\item $X_{2} \oplus Y_{2}=1 \oplus \coldue{1}{2}$ is $\tau$-rigid, $X_{2} \in\add T$ and $Y_{2} \in\add T'$.
\item $X_{3} \oplus Y_{2}=2 \oplus \coldue{1}{2}$ is $\tau$-rigid, $Y_{2} \in\add T', X_{3}=P$ 
and $\left(Y_{2}, X_{3}\right)=\left(\coldue{1}{2}, 2\right)$ is not $\tau$-rigid. 
\end{enumerate}
Hence, we deduce from (1), (2) and (3) that 
$2 \notin F(1), 2 \notin F(2)$ and $2 \in F(3)\setminus H(3)$. It follows
that $2 \notin H(1) \cup H(2) \cup H(3)$. On the other hand,
the following facts hold: 
\begin{enumerate}[label=(4)]
\item[(4)] $\tonde{X_{1}, Y_{3}}=\tonde{\coldue{1}{3}, 3}$ is not $\tau$-rigid and $Y_3=P'$.
\item[(5)] $X_{2} \oplus Y_{1}=1 \oplus 2$ is not $\tau$-rigid.
\item[(6)] $X_{3}=2=Y_{1}$.
\end{enumerate}
Consequently, we have $3 \in H(1), 1 \in H(2)$ and $1 \in H(3)$, and so $H(i) \neq \emptyset$ for any $i$.
\end{proof}
Keeping all the notation of Theorem 2.1, the next example shows that the bijection $X_{i} \mapsto Y_{s(i)}$ does not necessarily send projective--injective modules to projective--injective modules.

\begin{ex} There exist support $\tau$-tilting pairs $\left(T, P\right)=\left(X_{1}, X_{2}\right),\left(T', P' \right)=\left(Y_{1}, Y_{2}\right)$ and $s \in S_{2}$ as in the hypotheses of Theorem 2.1 such that $X_{2}$ is projective-injective, but $Y_{s(2)}$ has infinite projective and injective dimension.
\end{ex}
\begin{proof}[Construction] Let $A$ be the algebra given by the quiver $1 \underset{b}{\stackrel{a}{\rightleftarrows}} 2$ with relations $a b=0$ and $b a=0$. 
Let $X_{1}=1, X_{2}=\coldue{2}{1}, Y_{1}=2$ and $Y_{2}=\coldue{1}{2}$. 
Then the following facts hold:
\begin{enumerate}
\item $\left(Y_{1}, X_{2}\right)=\left(2, \coldue{2}{1}\right)$ is not $\tau$-rigid and $X_{2}=P$.
\item $\left(X_{1}, Y_{2}\right)=\left(1, \coldue{1}{2}\right)$ is not $\tau$-rigid and $Y_{2}=P'$.
\end{enumerate}
Then we deduce from (1) that $1 \in F(2)$, and we deduce from (2) that $2 \in F(1)$.

Consequently, the permutation $S=(1 2)$ satisfies the hypotheses of Theorem 2.1. 
Hence $X_{2}= \coldue{2}{1}$ is a projective-injective module, but $Y_{s(2)}=Y_{1}=2$ has infinite projective and injective dimension.
\end{proof}

With the notation of Theorem 2.1, the next examples show that the bijection $X_{i} \mapsto Y_{s(i)}$ does not always fix an indecomposable common projective summand $X_{2}$ of $T \oplus P$ and $T' \oplus P'$ even in very special cases, namely when $Y_{s(2)}$ is projective (Example 2.6) or injective (Example 2.7).

\begin{ex}
There exist support $\tau$-tilting pairs $\left(T,P\right)=\left(X_{1}, X_{2}\right)$ and $\left(T',P'\right)=\left(Y_{1}, Y_{2}\right)$ and $s \in S_{2}$ as in the hypotheses of Theorem 2.1 with the following properties:
\begin{enumerate}[(i)] 
\item $X_{2}$ and $Y_{s(2)}=Y_{2}$ are projective modules.
\item $X_{2}=Y_{1}$.
\end{enumerate}
\end{ex}
\begin{proof}[Construction]
Let $A$ be the algebra given by the quiver 
$$\xymatrix{ \ar@(ul,dl)[]_{a} {1} \ar[r]^{b} & {2}}$$ 
with relations 
$a^{2}=0$ and $b a=0$. Then we have $\tau\tonde{\coldue{1}{1}} =2$. Let $X_{1}=\coldue{1}{1}, X_{2}=2, Y_{1}=2$ and $Y_{2}=\coldue{1}{1\ \ 2}$. Then the following facts holds:
\begin{enumerate}[(1)] 
\item $X_{1} \oplus Y_{1}=\coldue{1}{1} \oplus 2$ is not $\tau$-rigid, $Y_{2}=P'$ and 
the pair $\tonde{X_{1},Y_{2}}= \tonde{\coldue{1}{1},\coldue{1}{1\ \ 2}}$ is not $\tau$-rigid. 
\item $X_{2}=2=Y_{1}, X_{2}=P$ and the pair $\left(Y_{2}, X_{2}\right)=\tonde{\coldue{1}{1\ \ 2},2}$ is not $\tau$-rigid.
\end{enumerate}
Hence, we deduce from (1) and (2) that $F(1)=\{1,2\}=F(2)$. Let $s$ be the identity of $S_{2}$. Then the modules $X_{2}=2=Y_{1}$ and $Y_{s(2)}=Y_{2}=\coldue{1}{1\ \ 2}$ satisfy conditions $(i)$ and $(ii)$.
\end{proof}
\begin{ex}
There exist support $\tau$-tilting pairs $\left(T_{1} P\right)=\left(X_{1} \oplus X_{2}, 0\right)$ and $\left(T_{1}^{\prime} P^{\prime}\right)=\left(Y_{1}, Y_{2}\right)$ and $s \in S_{2}$ as in the hypotheses of Theorem 2.1 with the following properties: 
\begin{enumerate}[(i)]
\item $X_{2}$ is projective and $Y_{s(2)}=Y_{1}$ is injective, but not projective.\\
\item $X_{2}=Y_{2}$.
\end{enumerate} 
\end{ex}
\begin{proof}[Construction]
 Let $A$ be the algebras given by the quiver 
$$\xymatrix{ 1 \ar[r]^{a} & {2}  \ar@(ur,dr)[]^{b} }$$
with relations $b a=0$ and $b^{2}=0$.  
Let $X_{1}=\coldue{1}{2}, X_{2}=\coldue{2}{2}, Y_{1}=1$ and $Y_{2}=\coldue{2}{2}$. 
Then $\tau(1)=\coldue{2}{2}$ and the following facts hold:
\begin{enumerate}[(1)]
\item $\tonde{X_{1},Y_{2}} = \tonde{\coldue{1}{2},\coldue{2}{2}}$ is not $\tau$-rigid and $Y_{2}=P'$. 
\item $X_{1} \oplus Y_{1}=\coldue{1}{2} \oplus 1$ is $\tau$-rigid, $X_{1}\in\add T$
and $Y_{1}=T'$.
\end{enumerate}
Hence we deduce from (1) and (2) that  $F(1)=\set{2}$. It follows that $s=(1\ 2)$, and so we have $Y_{s(2)}=Y_{1}=1$ and $X_{2}=\coldue{2}{2}=Y_{2}$. Hence (i) and (ii) hold.
\end{proof}
\bibliographystyle{amsplain}
\bibliography{destesupport}
\end{document}